\documentclass[11pt,letterpaper]{amsart}
\usepackage{amsmath,amssymb,amsthm}
\usepackage{enumerate}
\usepackage{tikz}

\DeclareMathOperator{\Tr}{Tr}
\newcommand{\C}{{\mathbb C}}
\newcommand{\F}{{\mathbb F}}
\newcommand{\Q}{{\mathbb Q}}
\newcommand{\Z}{{\mathbb Z}}
\newcommand{\card}[1]{\left|{#1}\right|}
\newcommand{\sums}[1]{\sum_{\substack{#1}}}
\newcommand{\mins}[1]{\min_{\substack{#1}}}
\newcommand{\floor}[1]{\left\lfloor{#1}\right\rfloor}
\newcommand{\grmul}[1]{{#1}^{\times}}
\newcommand{\Fu}{\grmul{F}}
\newcommand{\cFu}{|\grmul{F}|}
\newcommand{\Fp}{\F_p}
\newcommand{\Wfd}{W_{F,d}}
\newcommand{\Wfda}{\Wfd(a)}
\newcommand{\valp}{v_p}
\newcommand{\Zn}{\Z/n\Z}
\newcommand{\Zbn}{\Z/(b^n-1)\Z}
\newcommand{\Zpn}{\Z/(p^n-1)\Z}
\newcommand{\Ztn}{\Z/(3^n-1)\Z}
\newcommand{\wt}{w}
\newcommand{\wtbn}{w_{b,n}}
\newcommand{\wtpn}{w_{p,n}}
\newcommand{\wttn}{w_{3,n}}
\newcommand{\tlf}[3]{\langle {#1}, {#2}, {#3} \rangle}
\newcommand{\bs}{\bar s}
\newcommand{\bw}{\bar w}
\newcommand{\bbw}{\bar{\bar w}}
\newcommand{\bx}{\bar x}
\newcommand{\bbx}{\bar{\bar x}}
\newcommand{\by}{\bar y}
\newcommand{\bby}{\bar{\bar y}}
\newcommand{\bz}{\bar z}

\def\refsec#1{Section \ref{#1}}
\def\refprop#1{Proposition \ref{#1}}
\def\reflem#1{Lemma \ref{#1}}
\def\refcor#1{Corollary \ref{#1}}
\def\refconj#1{Conjecture \ref{#1}}
\def\refstep#1{Step \ref{#1}}
\def\reffig#1{Figure \ref{#1}}
\def\reftab#1{Table \ref{#1}}
\def\reffoot#1{footnote \ref{#1}}

\newtheorem{theorem}{Theorem}[section]
\newtheorem{proposition}[theorem]{Proposition}
\newtheorem{lemma}[theorem]{Lemma}
\newtheorem{corollary}[theorem]{Corollary}
\newtheorem{conjecture}[theorem]{Conjecture}
\theoremstyle{remark}
\newtheorem{step}{Step}

\title[Three-Valued Weil Sums]{Proof of a Conjectured Three-Valued Family of Weil Sums of Binomials}

\author{Daniel J. Katz}
\address{Department of Mathematics, California State University, Northridge, \: United States}

\author{Philippe Langevin}
\address{Institut de Math\'ematiques de Toulon, Universit\'e de Toulon, France}

\date{first version: 08 September 2014; this version: 17 March 2015}

\begin{document}

\begin{abstract}
We consider Weil sums of binomials of the form $W_{F,d}(a)=\sum_{x \in F} \psi(x^d-a x)$, where $F$ is a finite field, $\psi\colon F\to {\mathbb C}$ is the canonical additive character, $\gcd(d,|F^\times|)=1$, and $a \in F^\times$.
If we fix $F$ and $d$ and examine the values of $W_{F,d}(a)$ as $a$ runs through $F^\times$, we always obtain at least three distinct values unless $d$ is degenerate (a power of the characteristic of $F$ modulo $|F^\times|$).
Choices of $F$ and $d$ for which we obtain only three values are quite rare and desirable in a wide variety of applications.
We show that if $F$ is a field of order $3^n$ with $n$ odd, and $d=3^r+2$ with $4 r \equiv 1 \pmod{n}$, then $W_{F,d}(a)$ assumes only the three values $0$ and $\pm 3^{(n+1)/2}$.
This proves the 2001 conjecture of Dobbertin, Helleseth, Kumar, and Martinsen.
The proof employs diverse methods involving trilinear forms, counting points on curves via multiplicative character sums, divisibility properties of Gauss sums, and graph theory.
\end{abstract}

\maketitle

\section{Introduction}

We consider Weil sums of binomials of the form
\[
\Wfda=\sum_{x \in F} \psi(x^d-a x),
\]
where $F$ is a finite field of characteristic $p$ and order $q$, $\psi\colon F\to \C$ is the canonical additive character, $d$ is a positive integer with $\gcd(d,q-1)=1$, and $a \in F$.
These sums and their relatives are much-studied objects in number theory \cite{Kloosterman,Vinogradow,Davenport-Heilbronn,Karatsuba,Carlitz-1978,Carlitz-1979,Lachaud-Wolfmann,Katz-Livne,Coulter,Cochrane-Pinner-2003,Cochrane-Pinner-2011}, and arise in applications to digital sequence design, cryptography, coding theory, and finite geometry, as detailed in \cite[Appendix]{Katz}.

For classical applications in communications like radar or signal synchronization, one fixes $F$ and $d$ and considers the values of $\Wfda$ as $a$ runs through $\Fu$; we ignore $\Wfd(0)$, which is the Weil sum of the monomial $x^d$, and is trivially $0$.
In such situations, it is desirable that all the values of $\Wfd$ be as small in magnitude as possible.
It is easy to calculate (see \cite[Proposition 3.1]{Katz}) that 
\[
\sum_{a \in \Fu} \Wfda^2=q^2,
\] 
which means that some $|\Wfda| > \sqrt{q}$, and it is possible to find $F$ and $d$ such that no $|\Wfda|$ is much larger than $\sqrt{q}$.
Some of the best choices of $F$ and $d$ for this purpose have the property that the number of distinct values, that is, $\card{\{\Wfda: a \in \Fu\}}$ is small.
However, we do exclude the case of {\it degenerate} $d$, that is, where $d$ is congruent to a power of $p$ modulo $q-1$, for then $\psi(x^d)=\psi(x)$, and our sum degenerates to the Weil sum of the monomial $(1-a)x$, with $\Wfd(1)=q$  (the largest possible magnitude for a Weil sum) and $\Wfd(a)=0$ for $a\not=1$.

We say that $\Wfd$ is {\it $v$-valued} to mean that $\card{\{\Wfda: a \in \Fu\}}=v$.
The fundamental result about how many values $\Wfd$ takes is due to Helleseth \cite[Theorem 4.1]{Helleseth}.
\begin{theorem}[Helleseth, 1976]
$\Wfd$ is at at least three-valued if $d$ is nondegenerate.
\end{theorem}
Thus the smallest number of distinct values for an interesting Weil sum $\Wfd$ is three.
From 1966 to the present, only nine infinite families of $(F,d)$ pairs that produce three-valued Weil sums $\Wfd$ have been discovered; these are listed in \cite[Table 1]{Aubry-Katz-Langevin}.

This paper adds a tenth three-valued infinite family by proving the following conjecture \cite[Conjecture B]{DHKM}.
\begin{conjecture}[Dobbertin-Helleseth-Kumar-Martinsen, 2001]\label{Alice}
If $F$ is a finite field of order $q=3^n$ with $n$ odd and $n > 1$, and $d=3^r+2$ with $4 r \equiv 1 \pmod{n}$, then $\Wfd$ is three-valued with
\[
\Wfda = \begin{cases}
0 & \text{for $q-q/3-1$ values of $a \in \Fu$,} \\
+\sqrt{3 q} & \text{for $(q+\sqrt{3 q})/6$ values of $a \in \Fu$, and} \\
-\sqrt{3 q} & \text{for $(q-\sqrt{3 q})/6$ values of $a \in \Fu$.}
\end{cases}
\]
\end{conjecture}
The original conjecture used the exponent $d_o=2\cdot 3^{r_o}+1$ with $4 r_0 \equiv -1 \pmod{n}$ where we use $d$.
But note that our $d \equiv 3^r d_o \pmod{q-1}$, so that the canonical additive character has $\psi(x^d)=\psi(x^{d_0})$ for all $x \in F$, and so $\Wfd=W_{F,d_o}$.
Also note that the condition $4 r \equiv 1 \pmod{n}$ does indeed make $d=3^r+2$ coprime to $q-1=3^n-1$,\footnote{For there is some positive integer $a$ with $4 r=a n+1$, and then $\gcd(3^r+2,3^n-1)$ is a divisor of $\gcd(3^{4 r}-16,3^{a n+1}-3) = \gcd(3-16,3^{a n+1}-3)$, which is in turn a divisor of $13$.  Thus $\gcd(3^r+2,3^n-1) \mid 13$, and yet $3^r+2 \equiv 3, 5, \text{ or } 11 \pmod{13}$ for every $r$.}\label{Edward} so that the family of three-valued Weil sums described in the conjecture meets the conditions we set down at the beginning of this section.

The rest of this paper is organized as follows.
In \refsec{Fred}, we show that the proof of \refconj{Alice} can be deduced if one knows two things: the sum of fourth powers of the values $\Wfda$, and the extent of $3$-divisibility of these same values.
Accordingly, the sum of fourth powers is determined in \refsec{SECMOM}, and the $3$-divisibility is determined in Sections \ref{SECGEN}--\ref{SECFREE}.
After some facts from the general theory of divisibility of character sums in \refsec{SECGEN}, we present two proofs of the divisibility result that we need: a very short computer-assisted proof in \refsec{SECCOMP}, and a somewhat technical computer-free proof in \refsec{SECFREE}.

\section{Method of Proof}\label{Fred}

As in the Introduction, we assume that $F$ is a finite field, that $\psi\colon F \to \C$ is the canonical additive character, that $d$ is a positive integer with $\gcd(d,\cFu)=1$, and that
\[
\Wfda=\sum_{x \in F} \psi(x^d-a x)
\]
for $a \in F$.
Here we show that \refconj{Alice} can be deduced from two propositions, whose proofs constitute the remaining sections of this paper.
This way to a proof had been proposed in \cite[p.~1475]{DHKM} by the authors of \refconj{Alice}, who noted that they had made some progress with this program, but they did not present details of their partial results.

The first proposition we need entails an exact calculation of the fourth power moment of the Weil sum.
\begin{proposition}\label{Dorothy}
If $F$ is a finite field of order $q=3^n$ with $n$ odd, and $d=3^r+2$ with $\gcd(d,q-1)=\gcd(r,n)=1$, then 
\[
\sum_{a \in \Fu} \Wfda^4 = 3 q^3.
\]
\end{proposition}
The second proposition gives the $3$-divisibility of the values of the Weil sum.
\begin{proposition}\label{Elaine}
If $F$ is a finite field of order $q=3^n$ with $n$ odd, and $d=3^r+2$ with $4 r \equiv 1 \pmod{n}$, then $\Wfda$ is a rational integer divisible by $\sqrt{3 q}$ for each $a \in F$.
\end{proposition}
These combine to give a proof of \refconj{Alice} as follows.
\begin{theorem}
If $F$ is a finite field of order $q=3^n$ with $n$ odd, $n > 1$, and $d=3^r+2$ with $4 r \equiv 1 \pmod{n}$, then $W_{F,d}$ is three-valued with
\[
W_{F,d} = \begin{cases}
0 & \text{for $q-q/3-1$ values of $a \in \Fu$,} \\
+\sqrt{3 q} & \text{for $(q+\sqrt{3 q})/6$ values of $a \in \Fu$, and} \\
-\sqrt{3 q} & \text{for $(q-\sqrt{3 q})/6$ values of $a \in \Fu$.}
\end{cases}
\]
\end{theorem}
\begin{proof}
The first two power moments of the Weil sum are well known (see, e.g., \cite[Proposition 3.1]{Katz}) as
\begin{align}
\sum_{a \in \Fu} \Wfda & = q, \label{Albert} \\
\sum_{a \in \Fu} \Wfda^2 & = q^2. \label{Christopher}
\end{align}
Now note that Proposition \ref{Dorothy} applies since the condition $4 r \equiv 1 \pmod{n}$ clearly makes $\gcd(r,n)=1$ and also makes $\gcd(d,q-1)=1$ by \reffoot{Edward} in the Introduction.
Then \eqref{Christopher} and Proposition \ref{Dorothy} show that
\[
\sum_{a \in \Fu} \Wfda^2(W_{F,d}^2-3 q) = 0,
\]
and Proposition \ref{Elaine} shows that the individual terms of this sum are nonnegative.
Thus all terms must be zero, and so $\Wfda \in \{0,\pm \sqrt{3 q}\}$ for all $a \in \Fu$.
If we let $N_0$, $N_+$, and $N_-$ denote the number of $a \in \Fu$ such that $\Wfda$ equals $0$, $+\sqrt{3 q}$, and $-\sqrt{3 q}$, respectively, then the total count of $\Fu$, along with \eqref{Albert} and \eqref{Christopher}, gives the system
\begin{align*}
N_0 + N_+ + N_- & = q-1 \\
\sqrt{3 q} N_+ - \sqrt{3 q} N_- & = q \\
3 q N_+ + 3 q N_- & = q^2,
\end{align*}
whence we deduce the claimed frequencies.
\end{proof}

\section{Fourth Power Moment}
\label{SECMOM}
The purpose of this section is to prove Proposition \ref{Dorothy}, 
which requires us to compute precisely the fourth 
power moment of our Weil sum.
Throughout this section, we assume that $F$ is a finite 
field of characteristic $p$ and order $q=p^n$, 
and that $\Tr\colon F\to\Fp$ is the absolute trace.
We let $\epsilon\colon \Fp \to \C$ be the canonical 
additive character of $\Fp$, that is, $\epsilon(x)=\exp(2\pi i x/p)$, 
and we let $\psi\colon F \to \C$ be the canonical 
additive character of $F$, that is, $\psi(x)=\epsilon(\Tr(x))$.
We also assume that $d=2+p^r$ for some nonnegative integer $r$ 
such that $\gcd(d,q-1)=1$, and define the Weil sum as usual:
\[
\Wfda=\sum_{x \in F} \psi(x^d-a x).
\]
We use the abbreviation $\bx$ for $x^{p^r}$, so that $x^d=\bx x^2$.

If we consider $F$ as a $\Fp$-vector space with $\Fp$-basis $\beta_1,\ldots,\beta_n$, and expand $x \in F$ as $x=x_1 \beta_1+\cdots+x_n \beta_n$ with $x_1,\ldots,x_n \in \Fp$, then $\Tr(x^d)$ is a cubic form in $x_1,\ldots,x_n$ over $\Fp$.
This kind of object is considered in \cite{LS}, which inspired the method we use here.

We define a symmetric $\Fp$-trilinear form on $F$,
\begin{equation}\label{Therese}
\tlf{x}{y}{z}=\Tr(\bx y z + x \by z + x y \bz),
\end{equation}
and we express the fourth power of our Weil sum in terms of this form.
\begin{lemma}\label{Emily}
We have
\[
\sum_{a \in \Fu} \Wfda^4 = q \sum_{x,y,z} \epsilon(\tlf{x}{y}{x}+\tlf{x}{y}{y}+2 \tlf{x}{y}{z}).
\]
\end{lemma}
\begin{proof}
Since $\Wfd(0)=0$, we change nothing by summing $\Wfda$ over all $a \in F$, so 
\begin{align*}
\sum_{a \in \Fu} \Wfda^4 
& = \sum_{a,t,u,v,w \in F} \psi(t^d+u^d+v^d+w^d-a(t+u+v+w)) \\
& = q \sums{t,u,v,w \in F \\ t+u+v+w =0} \psi(t^d+u^d+v^d+w^d) \\
& = q \sum_{x,y,z \in F} \psi((x+y+z)^d-(x+z)^d-(y+z)^d+z^d),
\end{align*}
where we have reparameterized with $t=x+y+z$, $u=-(x+z)$, $v=-(y+z)$, and $w=z$ in the last step, and used the fact that our condition $\gcd(d,q-1)=1$ makes $d$ odd when we are in odd characteristic.
Now use the fact that $s^d=s^2 \bs$ to expand out $(x+y+z)^d-(x+z)^d-(y+z)^d+z^d$ to obtain
\[
2 x \bx y + x^2 \by + 2 x y \by + \bx y^2 +  2 \bx y z + 2 x \by z + 2 x y \bz,
\]
so that the trace of this quantity is $\tlf{x}{y}{x}+\tlf{x}{y}{y}+2\tlf{x}{y}{z}$, which completes the proof, since $\psi=\epsilon \circ \Tr$.
\end{proof}
If we fix $x$ and $y$, then $z\mapsto \tlf{x}{y}{z}$ is 
an $\Fp$-linear form.
Let the kernel $K$ be the set of $(x,y) \in F^2$ that make this the zero functional:
\[
K=\{(x,y) \in F^2 : \tlf{x}{y}{z}=0 \text{ for every } z \in F\}.
\]
Then a consequence of our previous result is that the fourth power moment is related to $\card{K}$.
\begin{corollary}\label{Lawrence}
If our field $F$ is of odd characteristic, then
\[
\sum_{a \in \Fu} \Wfda^4 = q^2 \card{K}.
\]
\end{corollary}
\begin{proof}
From \reflem{Emily}, we have
\[
\sum_{a \in \Fu} \Wfda^4 = q \sum_{(x,y) \in F^2} \epsilon(\tlf{x}{y}{x}+\tlf{x}{y}{y}) \sum_{z \in F} \epsilon(2 \tlf{x}{y}{z}).
\]
If $(x,y) \not\in K$, then $z\mapsto 2 \tlf{x}{y}{z}$ is a nontrivial $\Fp$-linear functional, so as $z$ runs through $F$, the value of $2 \tlf{x}{y}{z}$ runs through $\Fp$, taking each value equally often, thus making the sum over $z$ vanish.
So we can restrict our sum over $(x,y)$ to $K$ to get
\begin{align*}
\sum_{a \in \Fu} \Wfda^4 
& = q \sum_{(x,y) \in K} \epsilon(\tlf{x}{y}{x}+\tlf{x}{y}{y}) \sum_{z \in F} \epsilon(2 \tlf{x}{y}{z}) \\
& = q \sum_{(x,y) \in K} \epsilon(0+0) \sum_{z \in F} \epsilon(0) \\
& = q^2 \card{K},
\end{align*}
where we use the definition of $K$ in the middle step.
\end{proof}
Now it remains to compute the size of $K$.
First we find a useful characterization of $K$ as the set of $F$-rational points on a curve.
\begin{lemma}\label{Konrad}
We have $K=\{(x,y) \in F^2: \bbx \by+\bx \bby + x y=0\}$.
\end{lemma}
\begin{proof}
We note that $\Tr(\bs)=\Tr(s)$ for any $s \in F$, because $\Tr(s^p)=\Tr(s)$, which means that the definition \eqref{Therese} of our trilinear form is equivalent to
\begin{align*}
\tlf{x}{y}{z} 
& = \Tr(\overline{\bx y z} + \overline{x \by z}+x y \bz) \\
& = \Tr((\bbx \by+\bx \bby+x y) \bz),
\end{align*}
and since $\Tr$ is a nonzero $\Fp$-functional of $F$ and $z\mapsto \bar{z}$ is an automorphism of $F$, our kernel $K$ is the set of $(x,y)$ that make $\bbx \by+\bx \bby+x y=0$.
\end{proof}
\begin{lemma}\label{Mary}
If our field $F$ is of characteristic $p=3$ and order $q=3^n$ with $n$ odd, and if our exponent $d=2+3^r$ has $\gcd(r,n)=1$, then $\card{K}=3 q$.
\end{lemma}
\begin{proof}
From the expression for $K$ in \reflem{Konrad}, it is clear that all $(x,0)$ and $(0,y) \in F^2$ lie in $K$, thus accounting for $2 q-1$ points.
So it remains to show that there are $q+1$ points $(x,y) \in K$ with $x,y\not=0$, and we reparameterize the condition in \reflem{Konrad} using $x=w y$ to obtain
\[
(\bbw+\bw)(\bby \by) + w y^2=0,
\]
and so we want to show that $q+1$ points $(w,y)$ with $w,y\not=0$ satisfy this equation, or equivalently, we want to show that
\[
S = \{(w,y) \in (\Fu)^2 : y^{2-3^r-3^{2 r}} = -w^{3^r-1} (w^{3^{2 r}-3^r}+1)\},
\]
has $q+1$ elements.
Note that $\gcd(2-3^r-3^{2 r},q-1)=\gcd((1-3^r)(2+3^r),3^n-1)=\gcd((3^r-1) d,3^n-1)$, and recall that $d$ is coprime to $3^n-1$, so that our greatest common divisor is $3^{\gcd(r,n)}-1=2$.
Thus $\card{S}=\card{T}$, where
\[
T=\{(v,w) \in (\Fu)^2: v^2 = -w^{3^r-1} (w^{3^{2 r}-3^r}+1)\},
\]
so it suffices to show that $\card{T}=q+1$.
Note that $w^{3^{2 r}-3^r}+1$ is never $0$, because this would imply that $-1$ is a quadratic residue in $F$, which it is not, since $[F:\F_3]=n$ is odd.
We can now compute $\card{T}$ using the quadratic character $\eta$ of $F$.
\begin{align*}
\card{T}
& = \sum_{w \in F^*} \left(1+\eta(-w^{3^r-1}(w^{3^{2 r}-3^r}+1))\right) \\
& = (q-1) - \sum_{w \in F^*} \eta(w^{3^{2 r}-3^r}+1),
\end{align*}
and then note that $\gcd(3^{2 r}-3^r,q-1)=\gcd(3^r(3^r-1),3^n-1)=3^{\gcd(r,n)}-1=2$, so that
\begin{align*}
\card{T}
& = (q-1)-\sum_{u \in F^*} \eta(u^2+1) \\
& = q - \sum_{u \in F} \eta(u^2+1) \\
& = q + 1,
\end{align*}
where we use the well known \cite[Theorem 5.48]{Lidl-Niederreiter} 
evaluation of the last character sum.
\end{proof}
\refcor{Lawrence} and \reflem{Mary} together immediately prove Proposition \ref{Dorothy}: the fourth power moment of our Weil sum is $3 q^3$.
\section{Divisibility: General Remarks}\label{SECGEN}
It only remains to prove \refprop{Elaine}, which we repeat here for convenience.
\begin{proposition}[\refprop{Elaine}, repeated]
If $F$ is a finite field of order $q=3^n$ with $n$ odd, and $d=3^r+2$ with $4 r \equiv 1 \pmod{n}$, then $\Wfda$ is a rational integer divisible by $\sqrt{3 q}$ for each $a \in F$.
\end{proposition}
The fact that $\Wfda \in \Z$ for every $a \in F$ follows immediately from a result of Helleseth \cite[Theorem 4.2]{Helleseth}.
\begin{theorem}[Helleseth, 1976]
If $F$ is a finite field of characteristic $p$, then $\Wfda \in \Z$ for all $a \in F$ if and only if $d \equiv 1 \pmod{p-1}$.
\end{theorem}
To prove the result on divisibility, we use a well known technique that relies on Stickelberger's Theorem (or alternatively, one can use McEliece's Theorem).
To state the principle, we use the $p$-adic valuation, written $\valp$, for a prime $p \in \Z$, and we extend $\valp$ to $\Q(e^{2\pi i/p})$ so that $\valp(e^{2\pi i/p}-1)=1/(p-1)$. 
Also, for $b$ and $n$ positive integers, we use the $b$-ary weight function $\wtbn\colon \Zbn \to \Z$, which computes the sum of the digits in the $b$-ary expansion of an $a \in \Zbn$.
That is, if we write an element $a \in \Zbn$ as $a=\sum_{i \in \Zn} a_i b^i$ with the elements $b^i$ in the group $\Zbn$ and each coefficient $a_i \in \{0,1,\ldots,b-1\}\subseteq \Z$ with at least one $a_i < b-1$, then $\wtbn(a)=\sum_{i \in \Zn} a_i$.
\begin{proposition}\label{Anthony}
Let $F$ be of characteristic $p$ and order $p^n$, and let
\[
m=\mins{j \in \Zpn \\ j\not=0} \wtpn(j)+\wtpn(-d j),
\]
or equivalently,
\[
m=(p-1) n + \mins{j \in \Zpn \\ j\not=0} \wtpn(d j)-\wtpn(j).
\]
Then $\valp(\Wfda) \geq m/(p-1)$ for all $a \in F$, with equality for some $a \in F$.
\end{proposition}
\begin{proof}
The equivalence of the two definitions of $m$ comes from reparameterizing with $-j$ for $j$ in the minimization, and using the fact that if a nonzero $j \in \Zpn$ has $p$-ary expansion $\sum_{i \in \Zn} j_i p^i$, then the element $-j$ has $p$-ary expansion $\sum_{i \in \Zn} (p-1-j_i) p^i$, so that $\wtpn(-j)=(p-1) n - \wtpn(j)$.

Lemma 4.1 of \cite{Aubry-Katz-Langevin} tells us that
\begin{equation}\label{Gerald}
\min_{a \in F^*} \valp(\Wfda) =  \mins{\chi\in \widehat{F^*} \\ \chi\not=1} \valp(\tau(\chi) \tau(\bar{\chi}^d)),
\end{equation}
where $\widehat{F^*}$ is the group of multiplicative characters of $F$, with the principal character denoted by $1$, and for any $\chi \in \widehat{F^*}$, we have the Gauss sum
\[
\tau(\chi)=\sum_{a \in F^*} \psi(a) \chi(a).
\]
If we let $\omega \colon F^* \to \C$ be the Teichm\"uller character, then Stickelberger's Theorem \cite[Theorem 2.1]{Lang} tells us that for $j \in \Zpn$, we have $\valp(\tau(\omega^j))=\wtpn(-j)/(p-1)$.
Thus, if we express the nontrivial multiplicative characters of $F$ as powers of the Teichm\"uller character, i.e., $\chi=\omega^{-j}$ for $j \in \Zpn$ with $j\not=0$, then equation \eqref{Gerald} becomes $\min_{a \in F^*} \valp(\Wfda) =  m/(p-1)$,
which is the desired result on the $p$-adic valuation of $\Wfda$.
\end{proof}
Given $j \in \Zpn$, we use a modular add-and-carry method inspired by \cite{HX} to help compute the weights of $-d j$ and $d j$ that appear in the formulae in \refprop{Anthony}.
The basic result we need is a technical result related to \cite[Lemma 3]{HX}.
\begin{lemma}\label{Alan}
Let $b$ and $n$ be positive integers with $b > 1$.
Suppose that we have $s_i, t_i \in \Z$ for every $i \in \Zn$, such that $\sum_{i \in \Zn} s_i b^i \equiv \sum_{i \in \Zn} t_i b^i \pmod{b^n-1}$.
Then there is a unique collection of integers $\{c_i\}_{i \in \Zn}$ such that
\begin{equation}\label{Olivia}
s_i + c_{i-1} = t_i + b c_i,
\end{equation}
for all $i \in \Zn$: these are in fact
\begin{equation}\label{Harold}
c_i = \frac{1}{b^n-1} \sum_{j=0}^{n-1} (s_{j+i+1}-t_{j+i+1}) b^j,
\end{equation}
for $i \in \Zn$.
Furthermore
\begin{equation}\label{Eveline}
\sum_{i \in \Zn} c_i = \frac{1}{b-1} \sum_{i \in \Zn} (s_i-t_i).
\end{equation}
\end{lemma}
\begin{proof}
The $c_i$ defined in \eqref{Harold} are indeed integers, because the sum is congruent modulo $b^n-1$ to $b^{-(i+1)} \sum_{j \in \Zn} (s_j-t_j) b^j$, which vanishes modulo $b^n-1$ by assumption.
Replace $i$ in \eqref{Olivia} with $j+i+1$, multiply both sides by $b^j$, and then sum this for $j$ from $0$ to $n-1$.
Then rearrange and divide by $b^n-1$ to obtain \eqref{Harold}.
Conversely, replace $i$ with $i-1$ in \eqref{Harold}, and subtract from this $b$ times \eqref{Harold} (with $i$ unchanged) to obtain \eqref{Olivia}.
Finally, sum \eqref{Olivia} for all $i \in \Zn$, rearrange, and divide by $b-1$ to obtain \eqref{Eveline}.
\end{proof}

For the rest of this paper, we assume that $n$ is odd and $d=2+3^r$ where $4 r \equiv 1 \pmod{n}$, and we write $\wt$ for $\wttn$.
By \refprop{Anthony}, we will complete our proof of \refprop{Elaine} if we show that
\begin{equation}\label{Matilda}
\wt(x) +\wt(-d x) \geq n+1,
\end{equation}
or, equivalently, show that
\begin{equation}\label{Nancy}
n + \wt(d x)-\wt(x) > 0,
\end{equation}
for all nonzero $x \in \Ztn$.
Our computer-assisted proof in Section \ref{SECCOMP} verifies \eqref{Nancy}, while our computer-free proof in Section \ref{SECFREE} proves \eqref{Matilda}.
\section{Computer-Assisted Proof of Divisibility}
\label{SECCOMP}
In this section, we use a graph-theoretic formulation as in \cite{LL} to provide a computational verification of \eqref{Nancy} (which then secures \refprop{Elaine}) by means of the algorithms of Tarjan and Bellman-Ford.
We continue to assume that $n$ is odd, that $d=2+3^r$ with $4 r \equiv 1 \pmod{n}$, and we use $\wt(a)$ to denote the sum of the digits in the ternary expansion of $a \in \Ztn$.  (Thus $\wt(a)$ here is $\wttn(a)$ per the definition given just before \refprop{Anthony}.)

To verify \eqref{Nancy}, we let $x$ be a given nonzero element of $\Ztn$, and set $y=d x$, and then our goal is to show
\begin{equation}\label{Natasha}
n+\wt(y)-\wt(x) > 0.
\end{equation}

For each $i \in \Zn$, we let $x_i, y_i \in \{0,1,2\} \subseteq \Z$ such that $x=\sum_{i \in \Zn} x_i 3^i$ and $y=\sum_{i \in \Zn} y_i 3^i$.
Since $y= d x$ with $d=2+3^r$, we can also write $y=\sum_{i \in \Zn} (2 x_i+x_{i-r}) 3^i$.
Then by \reflem{Alan}, there are integers $c_i$ for $i \in \Zn$ such that
\begin{equation}\label{BASIC}
y_i + 3 c_i = 2 x_i+ x_{i-r} + c_{i-1}
\end{equation}
for every $i \in \Zn$.

We now set $X_i=x_{r i}$, $Y_i=y_{r i}$, and $C_i=c_{r i}$ for each $i \in \Zn$, and use the fact that $4 r \equiv 1 \pmod{n}$ to reparameterize \eqref{BASIC} with $i=r j$ to obtain
\begin{equation}\label{CARRY}
Y_j + 3 C_j = 2 X_j  + X_{j-1} + C_{j-4}.
\end{equation}
Note that $r$ no longer explicitly appears in our formula.
Since $Y_j \in \{0,1,2\}$ for every $j$, we see that
\begin{equation}\label{Roland}
C_j = \floor{\frac{2 X_j+X_{j-1}+C_{j-4}}{3}}.
\end{equation}

We sum \eqref{CARRY} over all $j \in \Zn$ to obtain
\begin{equation}\label{Henry}
\sum_{j \in \Zn} Y_j + 2 \sum_{j \in \Zn} C_j = 3 \sum_{j \in \Zn} X_j,
\end{equation}
and then note that $\sum_{j \in \Zn} X_j = \sum_{i \in \Zn} x_i = \wt(x)$ and $\sum_{j \in \Zn} Y_j=\wt(y)$ to get
\[
\sum_{j \in \Zn} C_j = \frac{3 \wt(x)-\wt(y)}{2}.
\]
Since $0 \leq \wt(x), \wt(y) < 2 n$, we see that there are $k,\ell \in \Zn$ with $C_k \geq 0$ and $C_\ell \leq 2$. 
Then one can use \eqref{Roland} and the fact that $X_i \in \{0,1,2\}$ for all $i \in \Zn$ to see that $C_{k+4} \geq 0$ and $C_{\ell+4} \leq 2$.
Continuing in this fashion (and recalling that $4 r \equiv 1 \pmod{n}$), we see that $C_j \in \{0,1,2\}$ for every $j \in \Zn$.

We again note that $\wt(x)=\sum_{j \in \Zn} X_j$ and $\wt(y)=\sum_{j \in \Zn} Y_j$, and employ \eqref{Henry} to see that \eqref{Natasha} (which is our goal) is equivalent to
\begin{equation}\label{Hubert}
\sum_{j \in \Zn} (1+2(X_j-C_j)) \geq 0,
\end{equation}
where the strict inequality has been replaced by a non-strict one inasmuch as the left hand side is always odd (since $n$ is odd).
So now our goal is to prove \eqref{Hubert}.

We consider a directed graph with $3^6$ vertices,
\[
(\xi,\gamma) = (\xi_0,\xi_1,\gamma_0,\gamma_1,\gamma_2,\gamma_3)\in\{0,1,2\}^6,
\]
with an edge $(\xi,\gamma)\to (\xi',\gamma')$ if and only if 
\begin{align*}
\xi'_0 & = \xi_1, \\ 
\gamma'_0 & = \gamma_1, \quad \gamma'_1=\gamma_2, \quad \gamma'_2=\gamma_3, \quad\text{and} \\
\gamma'_3 &= \floor{\frac{\xi_0 + 2\xi_1 +\gamma_0}{3}}.
\end{align*}
(Note that there is no constraint on $\xi'_1$.)
If we write the sextuple $T_j=(X_{j-1},X_j,C_{j-4},C_{j-3},C_{j-2},C_{j-1})$ for each $j \in \Zn$, then the sequence $T_0,T_1,\ldots,T_{n-1},T_0$ traces a directed cycle of length $n$ in our directed graph: the first four conditions for an edge are immediately satisfied by the structure of $T_j$ and $T_{j+1}$, while the last condition is satisfied because of \eqref{Roland}.
Furthermore, if we attach to each directed edge a cost
\[
\kappa((\xi,\gamma),(\xi',\gamma')) = 1 + 2 (\xi_1-\gamma_0),
\]
then the total cost for our directed cycle is equal to $\sum_{j \in \Zn} (1+2(X_j-C_j))$.
Thus to verify \eqref{Hubert} (which secures \refprop{Elaine}), it suffices to show that the graph does not contain any {\it absorbent circuit}, that is, a circuit of strictly negative cost.

The graph is of order $729$, with $2187$ edges. We apply  
Tarjan's algorithm  to split the graph into $258$ strongly 
connected components. All of them are trivial 
(singleton without circuit) 
except two components: a large one of size $471$, 
and the following of order $2$
\[
[0,2,2,0,2,0] \buildrel {\bf 0}\over\longleftrightarrow [2,0,0,2,0,2]
\]
On the large component, we apply the Bellman-Ford algorithm
to  prove the non-existence of an absorbent circuit.
Note that the running time is negligible.

\section{Computer-Free Proof of Divisibility}
\label{SECFREE}
In this section, we provide a proof of \eqref{Matilda} (which then secures \refprop{Elaine}) that does not use a computer.

We continue to assume that $n$ is odd, that $d=2+3^r$ with $4 r \equiv 1 \pmod{n}$, and we use $\wt(a)$ to denote the sum of the digits in the ternary expansion of $a \in \Ztn$.  (Thus $\wt(a)$ here is $\wttn(a)$ per the definition given just before \refprop{Anthony}.)

To verify \eqref{Matilda}, we let $x$ be a nonzero element of $\Ztn$ that makes $\wt(x)+\wt(-d x)$ as small as possible, and furthermore, among such minimizers, we choose an $x$ with $\wt(x)$ as small as possible.
We set $z=-d x$, and then our goal is to show
\begin{equation}\label{Yvette}
\wt(x)+\wt(z) \geq n+1.
\end{equation}
\begin{step}[Upper Limit on $\wt(x)+\wt(z)$]\label{Laura}
Let $a=1+3^{2 r}+3^{4 r} + \cdots + 3^{(n-3)r} \in \Ztn$, which has $\wt(a)=(n-1)/2$, and then note that $-da = 3^r+3^{3 r} + 3^{5 r} + \cdots + 3^{(n-2) r} + 2 \cdot 3^{(n-1) r}$ has $\wt(-d a)=(n+3)/2$, so that $\wt(a)+\wt(-d a)=n+1$.
Thus $\wt(x)+\wt(z) \leq n+1$, since we chose $x$ to minimize this sum of weights.
\end{step}
\begin{step}[The Five Surgeries]
A {\it surgery} is a modification to our $x$ that would change it to an $x'$ such that if $z'=-d x'$, then $\wt(x')+\wt(z') \leq \wt(x)+\wt(z)$, and yet $\wt(x') < \wt(x)$, thus contradicting our choice of $x$.
So our $x$ must not be susceptible to any surgeries.
In order to describe our surgeries, for each $i \in \Zn$, we let $x_i, z_i \in \{0,1,2\} \subseteq \Z$ such that $x=\sum_{i \in \Zn} x_i 3^i$ and $z=\sum_{i \in \Zn} z_i 3^i$.
We list five surgeries on \reftab{Sigmund}, with the conditions on $x$ and $z=-d x$ under which they can be performed: thus our $x$ and $z$ must not satisfy any of these conditions.
\begin{table}[ht]\caption{The Surgeries}\label{Sigmund}
\begin{center}
\begin{tabular}{c|c|c}
Surgery & Conditions Needed & Values of $x'$ and $z'$ \\
Number & (for any $i \in \Zn$) & \\ 
\hline
\hline
I & $x_i \geq 1$ & $x'=x-3^i$ \\
  & $z_i \geq 1$ & $z'=z-3^i+3^{i+1}+3^{i+r}$ \\
\hline
II & $x_i=2$ & $x'=x-2\cdot 3^i$ \\
   & $z_{i+r} \geq 1$ & $z'=z+3^i+3^{i+1}-3^{i+r}+3^{i+r+1}$ \\
\hline
III & $x_i=2$ & $x'=x-2\cdot 3^i+3^{i+r}$ \\
    & $z_{i+2 r} \geq 1$ & $z'=z+3^i+3^{i+1}-3^{i+2 r}$ \\
\hline
IV & $x_i=2$ & $x'=x-2\cdot 3^i+3^{i+r}-2\cdot 3^{i+2 r}+3^{i+3 r}$ \\
   & $x_{i+2 r}=2$ & $z'=z+3^i+3^{i+2 r+1}$ \\
\hline
V  & $x_i=2$ & $x'=x-2\cdot 3^i+3^{i+r}-3^{i+2 r}+3^{i+3 r}$ \\
   & $x_{i+2 r}\geq 1$ & $z'=z+3^i+3^{i+2 r}-3^{i+3 r}$ \\
   & $z_{i+3 r} \geq 1$ & 
\end{tabular}
\end{center}
\end{table}

We check $z'=-d x'$ in each proposed surgery: since $z=-d x$ and $d=2+3^r$, this amounts to checking that $z'-z=(2+3^r)(x-x')$.
For the five items in our table, we therefore check that
\begin{align*}
-3^i+3^{i+1}+3^{i+r} & = (2+3^r) 3^i, \\
3^i+3^{i+1}-3^{i+r}+3^{i+r+1} & = (2+3^r)2\cdot 3^i, \\
3^i+3^{i+1}-3^{i+2r} & = (2+3^r) (2\cdot 3^i-3^{i+r}), \\
3^i+3^{i+2r+1} & = (2+3^r) (2\cdot 3^i-3^{i+r}+2\cdot 3^{i+2r}-3^{i+3r}), \\
3^i+3^{i+2r}-3^{i+3r} & = (2+3^r) (2\cdot 3^i-3^{i+r}+3^{i+2r}-3^{i+3r}),
\end{align*}
the first three of which are easily verified to be true, and the last two are also readily verified once one recalls that $4 r \equiv 1 \pmod{n}$.

We now consider $\wt(x')$ and $\wt(z')$ for our surgeries.
On the one hand, if $x_i \geq a$, then $\wt(x-a\cdot 3^i)=\wt(x)-a$.
On the other hand, for any $j$, we have $\wt(x+3^j)\leq \wt(x)+1$, for the ternary expansion of $x+3^j$ is obtained by finding the position $k \in \Zn$ such that $x_j=\cdots=x_{k-1}=2$ and $x_k < 2$, and then replacing each of $x_j,\ldots,x_{k-1}$ with $0$, and replacing $x_k$ with $x_k+1$.
We apply these two principles to the estimation of $\wt(x')$ and $\wt(z')$, to obtain, respectively, for the five items in our table
\begin{align*}
\wt(x')=\wt(x)-1, \quad & \quad \wt(z')\leq\wt(z)+1, \\
\wt(x')=\wt(x)-2, \quad & \quad \wt(z')\leq\wt(z)+2, \\
\wt(x')\leq\wt(x)-1, \quad & \quad \wt(z')\leq\wt(z)+1, \\
\wt(x')\leq\wt(x)-2, \quad & \quad \wt(z')\leq\wt(z)+2, \\
\wt(x')\leq\wt(x)-1, \quad & \quad \wt(z')\leq\wt(z)+1,
\end{align*}
so that $\wt(x')+\wt(z')\leq\wt(x)+\wt(z)$ and $\wt(x') < \wt(x)$ in each case.
This completes the verification that the items in our table are indeed surgeries, and therefore our $x$ and $z$ cannot fulfil any set of conditions that would allow one of these surgeries to be performed.
\end{step}
\begin{step}[The Carry Sequence]
Since $d=2+3^r$, $x=\sum_{i\in\Zn} x_i 3^i$, and $z=\sum_{i\in\Zn} z_i 3^i$, we can write $d x+z=\sum_{i \in \Zn} (2 x_i+x_{i-r}+z_i) 3^i$.
Because $z=-d x$, we see that $d x + z = 0 = \sum_{i \in\Zn} 2 \cdot 3^i$, and so \reflem{Alan} shows us that there are integers $c_i$ for $i \in \Zn$ such that
\begin{equation}\label{Robert}
2 + 3 c_i = 2 x_i+ x_{i-r} + z_i + c_{i-1},
\end{equation}
for every $i \in \Zn$.
We call $\{c_i\}_{i \in \Zn}$ the {\it carry sequence} for $x$ and $z$.
\end{step}
\begin{step}[Reparameterization]
We now set $X_i=x_{r i}$, $Z_i=Z_{r i}$, and $C_i=c_{r i}$ for each $i \in \Zn$, and use the fact that $4 r \equiv 1 \pmod{n}$ to reparameterize \eqref{Robert} with $i=r j$ to obtain
\begin{equation}\label{Andrew}
2 + 3 C_j = 2 X_j  + X_{j-1} + Z_j + C_{j-4}.
\end{equation}
Note that $r$ no longer explicitly appears in our formula.

Using the reparameterization $j=r i$, we translate the conditions from \reftab{Sigmund}, which must not hold for the ternary digits $x_i$ and $z_i$ of our $x$ and $z$, into equivalent conditions that must not hold for our $X_j$ and $Z_j$, and list these forbidden conditions on \reftab{Michael} below.
\begin{table}[ht]\caption{Forbidden Conditions on $X_j$ and $Z_j$}\label{Michael}
\begin{center}
\begin{tabular}{c|c}
Surgery Number & Disallows the Condition \\
\hline
I & $X_j \geq 1$ and $Z_j \geq 1$ \\
II & $X_j=2$ and $Z_{j+1} \geq 1$ \\
III & $X_j=2$ and $Z_{j+2} \geq 1$ \\
IV & $X_j=2$ and $X_{j+2}=2$ \\
V & $X_j=2$, $X_{j+2}\geq 1$, and $Z_{j+3} \geq 1$ 
\end{tabular}
\end{center}
\end{table}
\end{step}
\begin{step}[Limitation on Carries]
We rearrange \eqref{Andrew} to obtain
\begin{equation}\label{Beatrice}
C_j = \frac{2 X_j+X_{j-1}+Z_j+C_{j-4}-2}{3}.
\end{equation}
Now note that $X_i, Z_i \in \{0,1,2\}$ for all $i \in \Zn$, and also note from \reftab{Michael} that Surgery I prevents $X_i$ and $Z_i$ from simultaneously being nonzero, so that
\begin{equation}\label{Dante}
\frac{C_{j-4}-2}{3} \leq C_j \leq \frac{C_{j-4}+4}{3}.
\end{equation}

Now sum \eqref{Andrew} over all $j \in \Zn$ to obtain
\[
2 n + 2 \sum_{j \in \Zn} C_j = 3 \sum_{j \in \Zn} X_j + \sum_{j \in \Zn} Z_j,
\]
and then note that $\sum_{j \in \Zn} X_j = \sum_{i \in \Zn} x_i = \wt(x)$ and $\sum_{j \in \Zn} Z_j=\wt(z)$ to get
\[
\sum_{j\in \Zn} C_j = \frac{3 \wt(x)+\wt(z)-2 n}{2}.
\]
Since $x\not=0$ and $d$ is coprime to $3^n-1$ (see \reffoot{Edward} in the Introduction), we know that $z=-d x\not=0$, so $\wt(x)$ and $\wt(z)$ are strictly positive.
Therefore \refstep{Laura} implies that $4 \leq 3\wt(x)+\wt(z) \leq 3 n+1$, and so 
\[
-(n-2) \leq \sum_{j\in\Zn} C_j \leq \frac{n+1}{2},
\]
so that there are $k, \ell \in \Zn$ such that $C_k \geq 0$ and $C_\ell \leq 1$.
Then \eqref{Dante} shows that $C_{k+4} \geq 0$ and $C_{\ell+4} \leq 1$.
Continuing in this fashion (and recalling that $4 r \equiv 1 \pmod{n}$), we see that $C_j \in \{0,1\}$ for all $j \in \Zn$.

For $j\in\Zn$, we say that position $j$ {\it gives a carry} when $C_j=1$, and we say that position $j$ {\it receives a carry} when $C_{j-4}=1$.
\end{step}
\begin{step}[Motifs]\label{Celeste}
For each $j \in \Zn$, we have $C_j \in \{0,1\}$, so relation \eqref{Beatrice} shows that there are only four possibilities for the sum $A_j=2 X_j + X_{j-1} + Z_j$: (i) $A_j=1$ and position $j$ receives but does not give a carry, (ii) $A_j=2$ and position $j$ neither receives nor gives a carry, (iii) $A_j=4$ and position $j$ both receives and gives a carry, (iv) $A_j=5$ and position $j$ does not receive but does give a carry.

These four possibilities can be realized only nine different ways at a given position $j$, and we call these possibilities the nine {\it motifs}, and list them on \reftab{Gabriel} below.
Note that some ways of achieving a sum $A_j=4$ or $5$ are omitted, as they would require the conditions forbidden by Surgery I or II (see \reftab{Michael}).
The name of each motif begins with a number that indicates the value of $A_j=2 X_j+X_{j-1}+Z_j$ when that motif occupies position $j$.
\begin{table}[ht]\caption{The Motifs}\label{Gabriel}
\begin{center}
\begin{tabular}{c||c|c||c||c|c}
Motif at Position $j$ & $X_{j-1}$ & $X_j$ & $Z_j$ & $C_{j-4}$ & $C_j$ \\ 
\hline
\hline
1A & 0 & 0 & 1 & 1 & 0 \\
1B & 1 & 0 & 0 & 1 & 0\\
\hline
2A & 0 & 0 & 2 & 0 & 0\\
2B & 0 & 1 & 0 & 0 & 0\\
2C & 1 & 0 & 1 & 0 & 0\\
2D & 2 & 0 & 0 & 0 & 0\\
\hline
4A & 0 & 2 & 0 & 1 & 1\\
4B & 2 & 1 & 0 & 1 & 1\\
\hline
5A & 1 & 2 & 0 & 0 & 1\\
\end{tabular}
\end{center}
\end{table}
\end{step}
\begin{step}[Rules of Succession]
The presence of a given motif at position $j$ requires particular values of both $X_{j-1}$ and $X_j$.
This constrains which motifs can precede or follow each other.
For example, motif 1B at position $j$ requires $X_{j-1}=1$ and $X_j=0$, which means that the motif at position $j-1$ can only be 2B or 4B, and the motif at position $j+1$ can only be 1A, 2A, 2B, or 4A.
We construct a directed graph that has the nine motifs as its vertices, and there is a directed edge from motif $M$ to motif $N$ if and only if the compatibility condition for $X_j$ allows motif $N$ to occupy position $j+1$ when motif $M$ occupies position $j$.

We depict the directed graph in \reffig{Terrence} below, but omit certain edges to avoid clutter.
In particular, the motifs from \reftab{Gabriel} that have $X_{j-1}=0$ are called {\it starting motifs} and are marked with a $*$, while the motifs from \reftab{Gabriel} with $X_j=0$ are called {\it ending motifs} and are marked with $\dag$, and to fill in the edges that we neglected to draw, one would draw a directed edge from every ending motif to every starting motif.  (Note that these are the only directed edges that emanate from ending motifs, and the only directed edges that enter starting motifs.)
We have drawn the edge from 4B to 5A with a dotted line: the reason for this will become apparent in \refstep{Felicity}.
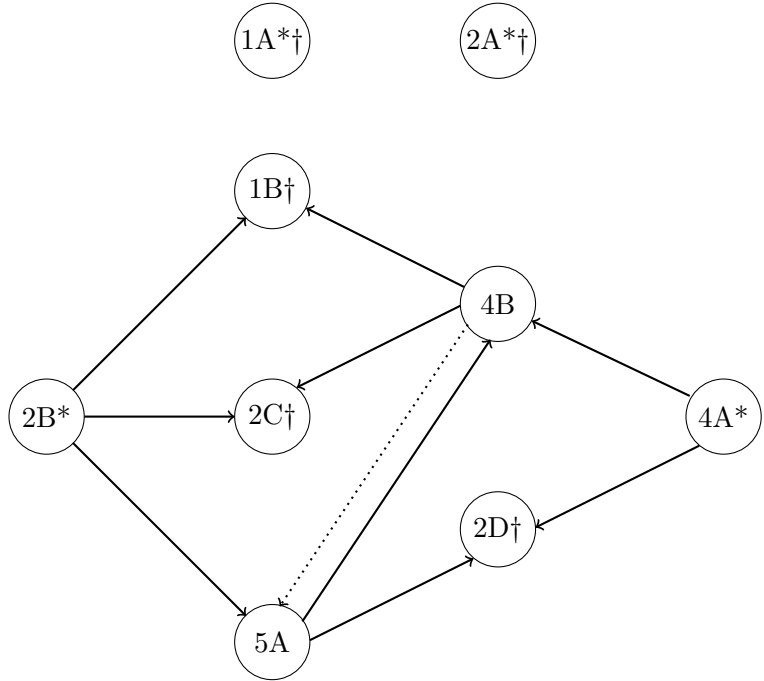
\begin{figure}[ht]\caption{Directed Graph Illustrating Succession of Motifs}\label{Terrence}
\begin{center}
\begin{tikzpicture}
\node (1A) at (0,-1) {1A*\dag};
\node (1B) at (0,-3) {1B\dag};
\node (2C) at (0,-6) {2C\dag};
\node (5A) at (0,-9) {5A};
\node (2A) at (3,-1) {2A*\dag};
\node (4B) at (3,-4.5) {4B};
\node (2D) at (3,-7.5) {2D\dag};
\node (2B) at (-3,-6) {2B*};
\node (4A) at (6,-6) {4A*};
\draw (1A) circle(0.5);
\draw (1B) circle(0.5);
\draw (2A) circle(0.5);
\draw (2B) circle(0.5);
\draw (2C) circle(0.5);
\draw (2D) circle(0.5);
\draw (4A) circle(0.5);
\draw (4B) circle(0.5);
\draw (5A) circle(0.5);
\draw [->,thick] (-2.65,-5.65)--(-0.35,-3.35); 
\draw [->,thick] (-2.5,-6)--(-0.5,-6); 
\draw [->,thick] (-2.65,-6.35)--(-0.35,-8.65);
\draw [->,thick] (0.5,-8.975)--(2.675,-7.8875);
\draw [->,thick] (2.55,-4.275)--(0.45,-3.225);
\draw [->,thick] (2.5,-4.525)--(0.325,-5.6125);
\draw [->,thick,dotted] (2.6,-4.775)--(0.1,-8.525);
\draw [->,thick] (0.4,-8.725)--(2.9,-4.975);
\draw [->,thick] (5.55,-5.725)--(3.45,-4.725);
\draw [->,thick] (5.675,-6.3875)-- (3.5,-7.475);
\end{tikzpicture}
\end{center}
\end{figure}
\end{step}
\begin{step}[Rule of Double Succession]\label{Bernard}
The conditions forbidden by Surgeries III and IV (see \reftab{Michael}) constrain which motif may follow two positions after another: a motif 4A or 5A at position $j$ cannot be succeeded by a motif 1A, 2A, 2C, 4A, or 5A at position $j+2$.
\end{step}
\begin{step}[The Forbidden Edge]\label{Felicity}
In \reffig{Terrence}, we have drawn the edge from motif 4B to motif 5A with a dotted line to indicate that we are actually forbidden to traverse it.
For we must have either motif 4A or 5A immediately before any occurrence of motif 4B, and if we proceed thence to motif 5A via the dotted edge, this will cause the motif 4A or 5A to be two positions before the motif 5A, contradicting the Rule of Double Succession in \refstep{Bernard}.
\end{step}
\begin{step}[The Ten Sequences]
Let us examine our directed graph after discarding the forbidden edge.
We define a {\it sequence of motifs}, or just a {\it sequence} to be a finite succession of motifs such that the first motif is a starting motif, the last motif is an ending motif, and all the other motifs are neither starting nor ending motifs.
(Note that a sequence can have a single motif, so long as that motif is both starting and ending.)
If we start at any vertex and follow directed edges, then no matter how we do it, we arrive at an ending motif in zero to three steps.
Similarly, if we start at any vertex and follow directed edges backwards, we arrive at a starting motif in zero to three steps.
Furthermore,  the only motif that can follow an ending motif is a starting motif, and the only motif that can precede a starting motif is an ending motif: these transitions correspond to the edges we did not draw explicitly in \reffig{Terrence}.
Thus our cyclic sequence of the $n$ motifs, for the $n$ positions $j \in \Zn$ is really a cyclic concatenation of sequences.

One can easily find all the possible sequences from \reffig{Terrence} by beginning at starting vertices (marked with $*$), following directed edges (but not the dotted forbidden edge), and arriving at ending vertices (marked with $\dag$).
We list all sequences on \reftab{Daphne} below.
\begin{table}[ht]\caption{The Sequences}\label{Daphne}
\begin{center}
\begin{tabular}{c|c}
Sequence Name & Constituent Motifs \\
\hline
S1 & 1A \\
S2 & 2A \\
S3 & 2B--1B \\
S4 & 2B--2C \\
S5 & 2B--5A--2D \\
S6 & 2B--5A--4B--1B \\
S7 & 2B--5A--4B--2C \\
S8 & 4A--2D \\
S9 & 4A--4B--1B \\
S10 & 4A--4B--2C
\end{tabular}
\end{center}
\end{table}

As mentioned before, the $n$ motifs for the $n$ positions $j\in\Zn$ make up a cyclic concatenation of sequences.
When we speak of one sequence {\it preceding} or {\it following} another, we are always proceeding cyclically around the $n$ positions.
If a single sequence were to account for all $n$ positions, it would therefore both precede and follow itself.
\end{step}
\begin{step}[Rule of Quadruple Succession]\label{Geoffrey}
If motif $M$ occupies position $j$ and motif $N$ occupies position $j+4$, any carry given by motif $M$ will be received by motif $N$.
As mentioned in \refstep{Celeste}, our motifs are named to help us keep track of carries: (i) motifs whose names begin in 4 or 5 give carries, while those whose names begin with 1 or 2 do not, and (ii) motifs whose names begin with 1 or 4 receive carries, while those whose names begin with 2 or 5 do not.
This leads to the following Rule of Quadruple Succession, summarized on \reftab{Alexander}.
\begin{table}[ht]\caption{Rule of Quadruple Succession}\label{Alexander}
\begin{center}
\begin{tabular}{c|c}
Four Places after a Motif & Must be a Motif \\ 
whose Name begins with  & whose Name begins with \\
\hline
1 or 2 & 2 or 5 \\
4 or 5 & 1 or 4
\end{tabular}
\end{center}
\end{table}
\end{step}
\begin{step}[Special Rule of Four]\label{Francois}
Suppose that we have the following succession of motifs: M--2D--2B--2C, where M is either 4A or 5A.
If M occupies position $j$, then we would have $X_j=2$ (from M), $X_{j+2}= 1$ (from 2B), and $Z_{j+3}=1$ (from 2C).
Surgery V (see \reftab{Michael}) forbids this, so such a succession of four motifs is prohibited.
This, in particular, prohibits us from having the following concatenations of sequences: S8--S4 or S5--S4.
\end{step}
\begin{step}[Elimination of S7, S10]
The sequences S7 and S10 violate the Rule of Double Succession (see \refstep{Bernard}), so they do not actually occur.
\end{step}
\begin{step}[Elimination of S6]
If S6 occurs, the Rule of Quadruple Succession (see \refstep{Geoffrey}) demands that it be preceded by a sequence whose last motif begins with a 4 or 5, but there is no such thing.
So S6 does not actually occur.
\end{step}
\begin{step}[Elimination of S5]
S5 cannot be followed by S1, S2, S8, or S9 by the Rule of Double Succession, nor by S3 or S5 by the Rule of Quadruple Succession, nor by S4 by the Special Rule of Four (\refstep{Francois}).
As S6, S7, and S10 have been eliminated, S5 can have no successor, so it never occurs.
\end{step}
\begin{step}[Elimination of S8]
S8 cannot be followed by S1, S2, S8, or S9 by the Rule of Double Succession, nor by S4 by the Special Rule of Four (\refstep{Francois}).  
If S8 were followed by S3, then the S8--S3 would be required by the Rule of Quadruple Succession to be preceded by a sequence whose last motif begins with a 4 or 5, but no such thing exists.
As S5, S6, S7, and S10 have been eliminated, S8 can have no successor, so it never occurs.
\end{step}
\begin{step}[Elimination of S9]
S9 cannot be preceded by S3 or S4 by the Rule of Quadruple Succession.
Nor can S9 be preceded by an S1 or S2, because then we would need the S1--S9 or S2--S9 to be preceded by a sequence whose last motif begins with a 4 or 5, but no such thing exists.
As S5, S6, S7, S8, and S10 have already been eliminated, S9 must always be preceded by S9.
But we cannot have all $n$ positions accounted for by S9 sequences, for then all positions would receive a carry, but not all of them would give a carry.
Thus S9 cannot actually occur.
\end{step}
\begin{step}[Elimination of S1, S3]
We have eliminated all sequences except S1, S2, S3, and S4.
Note that none of these sequences contains a motif that gives a carry.
So we cannot employ any motifs that receive a carry, and since S1 and S3 contain such motifs, they are eliminated.
\end{step}
\begin{step}[Conclusion]
Our $n$ positions are taken up entirely with sequences S2 and S4.
Since each S2 occupies one position, while each S4 occupies two positions, say that there are $k \leq (n-1)/2$ instances of S4 and $n-2 k$ instances of S2.
This, in turn, means that we have $n-2 k$ instances of motif 2A, and $k$ instances each of motifs 2B and 2C.

Recall that $X_j=x_{r j}$ and $Z_j=z_{r j}$ for every $j\in\Zn$.
Since $X_j=0$, $1$, or $0$, respectively, when motif 2A, 2B, or 2C occupies position $j$, we have $\wt(x)=\sum_{i \in \Zn} x_i = \sum_{j \in \Zn} X_j= k$.
Since $Z_j=2$, $0$, or $1$, respectively, when motif 2A, 2B, or 2C occupies position $j$, we have $\wt(z)=\sum_{i \in \Zn} z_i = \sum_{j \in \Zn} Z_j= 2(n-2 k)+k=2 n-3 k$.
So $\wt(x)+\wt(z)=2 n-2 k$, and since $k\leq (n-1)/2$, this means that $\wt(x)+\wt(z) \geq n+1$, which is \eqref{Yvette}, i.e., what we wanted to show.
\end{step}
\section*{Acknowledgements}
The first author was supported in part by a Research, Scholarship, and Creative Activity Award from California State University, Northridge.
The first author was also supported in part by the Institut de Math\'ematiques de Toulon at Universit\'e de Toulon as a visiting professor.
The authors thank an anonymous referee for helpful corrections to the manuscript.

\end{document}